\documentclass{amsart}
\usepackage[T1]{fontenc}
\usepackage{lmodern}
\usepackage[latin1]{inputenc}
\usepackage{amsmath}
\usepackage{amssymb}
\usepackage{amsthm}
\usepackage[enableskew]{youngtab}
\usepackage{hyperref}
\usepackage{breakurl}
\usepackage{color}
\usepackage{paralist}
\usepackage{setspace}
\newcommand{\NN}{\ensuremath{\mathbb N}}

\theoremstyle{plain}
\newtheorem{Th}{Theorem}[section]
\newtheorem{Prop}[Th]{Proposition}
\newtheorem{Co}[Th]{Corollary}
\newtheorem{Lem}[Th]{Lemma}
\newtheorem{Ex}[Th]{Example}
\theoremstyle{definition}
\newtheorem{Def}[Th]{Definition}
\newtheorem*{Ack}{Acknowledgement}
\theoremstyle{remark}
\newtheorem*{Rem}{Remark}

\newtheorem*{Con}{\underline{\bf{Convention}}}

\numberwithin{equation}{section}
\newcommand{\meins}{1'}
\newcommand{\mzwei}{2'}

\newcommand{\msechs}{6'}

\newcommand{\deins}{\mathbf 1}
\newcommand{\dzwei}{\mathbf 2}

\newcommand{\dmzwei}{\mathbf 2'}
\newcommand{\dmdrei}{\mathbf 3'}
\newcommand{\none}{\textcolor{white}{.}}

\usepackage{color}

\begin{document}
\title{Classification of $Q$-homogeneous skew Schur $Q$-functions}
\author{Christopher Schure}
\address{\scshape Institut für Algebra, Zahlentheorie und Diskrete Mathematik, Leibniz Universität Hannover, Welfengarten 1, D-30167 Hannover; schure@math.uni-hannover.de}
\begin{abstract}
We classify the $Q$-homogeneous skew Schur $Q$-functions, i.e., those of the form $Q_{\lambda/\mu} = k \cdot Q_{\nu}$.
On the way we develop new tools that are useful also in the context of other classification problems for skew Schur $Q$-functions.
\end{abstract}

\maketitle

\section{Introduction}
Schur functions form an important and well studied basis of the algebra of symmetric functions.
They appear in the study of the representations of the symmetric groups and the general linear groups.
In the subalgebra generated by the odd power sums, the Schur $Q$-functions form a basis which is less well studied but share similarities to the basis of Schur functions.
Stembridge showed in \cite{Stembridge} that Schur $Q$-functions play an analogous role for the irreducible spin characters of the symmetric groups as the Schur functions do for the ordinary irreducible characters of the symmetric groups.
In the decomposition of skew Schur $Q$-functions into Schur $Q$-functions, the shifted Littlewood-Richardson coefficients appear.
The very same coefficients appear in the decomposition of reduced Clifford products of spin characters into spin characters.
In the aforementioned paper, Stembridge proved a shifted analogue of the Littlewood-Richardson rule to calculate these coefficients.
We will classify the skew Schur $Q$-functions that are some multiple of a Schur $Q$-function by using this version of the Littlewood-Richardson rule (a different version is given by Cho in \cite{Cho}).
Our classification is an extension of
Salmasian's classification of skew Schur $Q$-functions that are equal to some Schur $Q$-function in \cite{Salmasian}.\par
Related classification results were published in the last years.
In \cite{Stembridge2}, Stembridge classified the multiplicity-free products of Schur functions.
Bessenrodt classified the $P$-multiplicity-free Schur $P$-functions (some multiple of the Schur $Q$-functions) in \cite{Bessenrodt} which is a shifted analogue of Stembridge's result.
A classification of multiplicity-free skew Schur functions can be found in \cite{Gutschwager} by Gutschwager and independently in \cite{ThomasYong} by Thomas and Yong in the context of Schubert calculus.
Bessenrodt and Kleshchev classified homogeneous skew Schur functions (that is, some multiple of a Schur function or a Schur function) in \cite{BessenrodtKleshchev}.
Our result is a shifted analogue of this result.\par
The structure of this note is as follows.
In the second section we state required definitions and known properties as well as a new criterion for amenable tableaux (which are the tableaux that are counted in Stembridge's version of the shifted Littlewood-Richardson rule).
In the third section we develop the tools required for the proof of the desired classification which includes the decomposition of specific skew Schur $Q$-functions into Schur $Q$-functions as well as some properties of the amenable tableaux with the lexicographically largest content and a bijection on skew diagrams that leaves their corresponding skew Schur $Q$-function unchanged.
In the last section we first examine the case of skew Schur $Q$-functions indexed by disconnected shifted diagrams before tackling the case of skew Schur $Q$-functions indexed by connected shifted diagrams to obtain our result, Theorem \ref{homogeneous}.\\

\section{Preliminaries}
\subsection{Partitions, diagrams and tableaux}
We will use notation compatible to \cite{Salmasian} and \cite{Stembridge}.\par
A \textbf{partition} is a tuple $\lambda = (\lambda_1, \lambda_2, \ldots, \lambda_n)$ where $\lambda_j \in \NN$ for all $1 \leq j \leq n$ and $\lambda_i \geq \lambda_{i+1} > 0$ for all $1 \leq i \leq n-1$.
The \textbf{length} of $\lambda$ is $\ell(\lambda) := n$.
A partition $\lambda$ is called a partition of $k$ if $|\lambda| := \lambda_1 + \lambda_2 + \ldots + \lambda_{\ell(\lambda)} = k$ where $|\lambda|$ is called the \textbf{size} of $\lambda$.
A \textbf{partition with distinct parts} is a partition $\lambda = (\lambda_1, \lambda_2, \ldots, \lambda_n)$ where $\lambda_i > \lambda_{i+1} > 0$ for all $1 \leq i \leq n-1$.
The set of partitions of $k$ with distinct parts is denoted by $DP_k$.
By definition, the empty partition $\emptyset$ is the only element in $DP_0$ and it has length $0$.
The \textbf{set of all partitions with distinct parts} is denoted by $DP := \bigcup_k DP_k$.
For some $\lambda \in DP$ the \textbf{shifted diagram} $D_{\lambda}$ is defined by $D_{\lambda} := \{(i,j) \mid 1 \leq i \leq \ell(\lambda), i \leq j \leq i+\lambda_i-1\}$.
\begin{Con}
In the following we will omit the adjective shifted.
This means whenever a diagram is mentioned it is always a shifted diagram.
\end{Con}
A diagram can be depicted as an arrangement of boxes where the coordinates of the boxes are interpreted in matrix notation.
\begin{Ex}
Let $\lambda = (6,5,2,1)$.
Then
$$D_{\lambda} = {\Yvcentermath1 \young(\none \none \none \none \none \none ,:\none \none \none \none \none ,::\none \none ,:::\none ).}$$
\end{Ex}
Let $\lambda, \mu \in DP$.
If $\ell(\mu) \leq \ell(\lambda)$ and $\mu_i \leq \lambda_i$ for all $1 \leq i \leq \ell(\mu)$ then the \textbf{skew diagram} is defined by $D_{\lambda/\mu} := D_{\lambda} \setminus D_{\mu}$.
Then the \textbf{size} is $|D_{\lambda/\mu}| = |D_{\lambda}| - |D_{\mu}|$.
Each edgewise connected part of the diagram is called a \textbf{component}.

The number of components of a diagram $D$ is denoted by $comp(D)$.
If $comp(D) = 1$ the diagram $D$ is called \textbf{connected}, otherwise it is called \textbf{disconnected}.\par
A \textbf{corner} of a diagram $D$ is a box $(x,y) \in D$ such that $(x+1,y), (x,y+1) \notin D$.

\begin{Con}
In the following, if components are numbered, the numbering is as follows: the first component is the leftmost component, the second component is the next component to the right of the first component etc.
\end{Con}

\begin{Ex}
Let $\lambda = (6,5,2,1)$ and $\mu = (4,3)$.
Then the diagram is
$$D_{\lambda/\mu} = {\Yvcentermath1 \young(::\none \none ,::\none \times ,\none \none ,:\times ).}$$
Its size is $|D_{\lambda/\mu}| = 7$.
The diagram $D_{\lambda/\mu}$ has two components where the first component consists of three boxes and the second component consists of four boxes.
The corners of $D_{\lambda/\mu}$ are the boxes marked $\times$.
\end{Ex}

We consider the alphabet $\mathcal{A} = \{1' < 1 < 2' < 2 < \ldots\}$.

\begin{Rem}
The letters $1, 2, 3, \ldots$ are called unmarked letters and the letters denoted by $1', 2', 3', \ldots$ are called marked letters.
For a letter $x$ of the alphabet $\mathcal{A}$ we denote the unmarked version of this letter with $|x|$.
\end{Rem}

\begin{Def}\label{tableaudef}
Let $\lambda, \mu \in DP$.
A \textbf{tableau} $T$ of shape $D_{\lambda/\mu}$ is a map $T: D_{\lambda/\mu} \rightarrow \mathcal{A}$ such that
\begin{enumerate}[a)]
	\item $T(i,j) \leq T(i+1,j)$, $T(i,j) \leq T(i,j+1)$ for all $i,j$,
	\item each column has at most one $k$ ($k = 1, 2, 3, \ldots$),
	\item each row has at most one $k'$ ($k' = 1', 2', 3', \ldots$).
\end{enumerate}
Let $c^{(u)}(T) = (c^{(u)}_1, c^{(u)}_2, \ldots)$ where $c^{(u)}_i$ denotes the number of all letters equal to $i$ in the tableau $T$ for each $i$.
Analogously, let $c^{(m)}(T) = (c^{(m)}_1, c^{(m)}_2, \ldots)$ where $c^{(m)}_i$ denotes the number of all letters equal to $i'$ in the tableau $T$ for each $i$.
Then the \textbf{content} is defined by $c(T) = (c_1, c_2, \ldots) := c^{(u)}(T) + c^{(m)}(T)$.
If there is some $k$ such that $c_k > 0$ but $c_j = 0$ for all $j > k$ then we omit all these $c_j$ from $c(T)$.
\end{Def}
\begin{Rem}
We depict a tableau $T$ of shape $D_{\lambda/\mu}$ by filling the box $(x,y)$ with the letter $T(x,y)$ for all $x, y$.
\end{Rem}
\begin{Ex}
Let $\lambda = (8,6,5,3,2)$ and $\mu = (5,2,1)$.
Then a tableau of shape $D_{\lambda/\mu}$ is
$$T = {\Yvcentermath1 \young(::\meins 12,\mzwei 224,2455,4\msechs 6,:67).}$$
We have $c(T) = (2,5,0,3,2,3,1)$.
\end{Ex}

\subsection{Schur $Q$-functions}
For $\lambda, \mu \in DP$ and a countable set of independent variables $x_1, x_2, \ldots$ the \textbf{skew Schur $Q$-function} is defined by
$$Q_{\lambda/\mu} := \sum_{T \in T(\lambda/\mu)}{x^{c(T)}}$$
where $T(\lambda/\mu)$ denotes the set of all tableaux of shape $D_{\lambda/\mu}$ and $x^{(c_1, c_2, \ldots, c_{\ell(c)})} := x_1^{c_1} x_2^{c_2} \cdots$ with $c_k := 0$ for $k > \ell(c)$.
If $D_{\mu} \nsubseteq D_{\lambda}$ then $Q_{\lambda/\mu} := 0$.
Since $D_{\lambda/\emptyset} = D_{\lambda}$, we denote $Q_{\lambda/\emptyset}$ by $Q_{\lambda}$.

\begin{Def}
Let $D$ be a diagram such that the $y^{\textrm{th}}$ column has no box but there are boxes to the right of the $y^{\textrm{th}}$ column and after shifting all boxes that are to the right of the $y^{\textrm{th}}$ column one box to the left we obtain a diagram $D_{\alpha/\beta}$ for some $\alpha, \beta \in DP$.
Then we call the $y^{\textrm{th}}$ column \textbf{empty} and the diagram $D_{\alpha/\beta}$ is obtained by removing the $y^{\textrm{th}}$ column.
Similarly, let $D$ be a diagram such that the $x^{\textrm{th}}$ row has no box but there are boxes below the $x^{\textrm{th}}$ row and after shifting all boxes that are below the $x^{\textrm{th}}$ row one box up and then all boxes of the diagram one box to the left we obtain a diagram $D_{\alpha/\beta}$ for some $\alpha, \beta \in DP$.
Then we call the $x^{\textrm{th}}$ row \textbf{empty} and the diagram $D_{\alpha/\beta}$ is obtained by removing the $x^{\textrm{th}}$ row.
\end{Def}

\begin{Def}
For $\lambda, \mu \in DP$ we call the diagram $D_{\lambda/\mu}$ \textbf{basic} if it satisfies the following properties for all $1 \leq i \leq \ell(\mu)$:
\begin{itemize}
	\item $D_{\mu} \subseteq D_{\lambda}$,
	\item $\ell(\lambda) > \ell(\mu)$,
	\item $\lambda_i > \mu_i$,
	\item $\lambda_{i+1} \geq \mu_i-1$.
\end{itemize}
This means that $D_{\lambda/\mu}$ has no empty rows or columns.
\end{Def}

For a given diagram $D$ let $\bar{D}$ be the diagram obtained by removing all empty rows and columns of the diagram $D$.
Since the restrictions of each box in a diagram is unaffected by removing empty rows and columns, there is a content-preserving bijection between tableaux of a given shape and tableaux of the respective shape obtained by removing empty rows and columns; thus we have $Q_D = Q_{\bar{D}}$.
Hence, we may restrict our considerations to partitions $\lambda$ and $\mu$ such that $D_{\lambda/\mu}$ is basic.\par
For some given skew diagram $D$ let the diagram obtained after removing empty rows and columns be $D_{\lambda/\mu}$ for some $\lambda, \mu \in DP$.
Then $Q_D$ is equal to the skew Schur $Q$-function $Q_{\lambda/\mu}$.

For a tableau $T$ of a diagram $D$ the \textbf{reading word} $w := w(T)$ is the word obtained by reading the rows from left to right beginning with the bottom row and ending with the top row.
The \textbf{length} $\ell(w)$ is the number of letters of the word $w$ and, thus, the number of boxes in $D$.\par
Let $(x(i),y(i))$ denote the box of the $i^{\textrm{th}}$ letter of the reading word $w(T)$.
It is the box that satisfies the property
$$|\{(u,v) \in D_{\lambda/\mu} \mid \text{either we have } u > x(i) \text{ or  we have } u = x(i) \text{ and } v \leq y(i)\}| = i.$$
For a reading word $w$ of length $n$ of a tableau the statistics $m_i(j)$ are defined as follows:
\begin{itemize}
	\item $m_i(0) = 0$ for all $i$.
	\item For $1 \leq j \leq n$ the statistic $m_i(j)$ is equal to the number of times $i$ occurs in the word $w_{n-j+1} \dots w_n$.
	\item For $n+1 \leq j \leq 2n$ we set $m_i(j) := m_i(n) + k(i)$ where $k(i)$ is the number of times $i'$ occurs in the word $w_1 \dots w_{j-n}$.
\end{itemize}
Note that $c^{(u)}_i = m_i(n)$ and $c^{(m)}_i = m_i(2n) - m_i(n)$.\par
As Stembridge remarked in \cite[before Theorem 8.3]{Stembridge}, the statistics $m_i(j)$ for some given $i$ can be calculated simultaneously by taking the word $w(T)$ and scan it first from right to left while counting the letters $i$ and afterwards scan it from left to right and adding the number of letters $i'$.
After the $j^{\textrm{th}}$ step of scanning and counting the statistic $m_i(j)$ is calculated.

\begin{Def}\label{amenable}
Let $k \in \NN$ and $w$ be a word of length $n$ consisting of letters from the alphabet $\mathcal{A}$.
The word $w$ is called $k$\textbf{-amenable} if it satisfies the following conditions:
\begin{enumerate}[a)]
	\item if $m_k(j) = m_{k-1}(j)$ then $w_{n-j} \notin \{k, k'\}$ for all $0 \leq j \leq n-1$,
	\item if $m_k(j) = m_{k-1}(j)$ then $w_{j-n+1} \notin \{k-1, k'\}$ for all $n \leq j \leq 2n-1$,
	\item if $j$ is the smallest number such that $w_j \in \{k', k\}$ then $w_j = k$,
	\item if $j$ is the smallest number such that $w_j \in \{(k-1)', k-1\}$ then $w_j = k-1$.
\end{enumerate}
A word $w$ is called \textbf{amenable} if it is $k$-amenable for all $k > 1$.
A tableau $T$ is called ($k$-)amenable if $w(T)$ is ($k$-)amenable.
\end{Def}
\begin{Rem}
Definition \ref{amenable} a) can be regarded as follows: Suppose that while scanning a word from right to left we have $m_k(j) = m_{k-1}(j)$ for some $j < n$.
Then the next letter that is scanned cannot be a $k'$ or $k$.\par
Similarly, \ref{amenable} b) can be regarded as follows: Suppose that while scanning a word from left to right we have $m_k(j) = m_{k-1}(j)$ for some $n \leq j < 2n$.
Then the next letter that is scanned cannot be a $k-1$ or $k'$.\par
Definition \ref{amenable} c) states that for a given amenable tableau $T$ the leftmost box of the lowermost row with boxes with entry from $\{k', k\}$ contains a $k$.
\end{Rem}

Clearly, Definition \ref{amenable} a) ensures $m_{k-1}(n) \geq m_k(n)$ for any $k$-amenable word of length $n$.
If $m_{k-1}(n) = m_k(n) > 0$ then while scanning the word from left to right a $k-1$ or a $k'$ is scanned before the first $(k-1)'$ in the word is scanned.
Thus, we have proven the following lemma.
\begin{Lem}\cite[Lemma 3.28]{Salmasian}\label{unmarked}
Let $w$ be a $k$-amenable word for some $k > 1$.
Let $n := \ell(w)$.
If $m_{k-1}(n) > 0$ then $m_{k-1}(n) > m_k(n)$.
\end{Lem}

We will use the shifted version of the Littlewood-Richardson rule by Stembridge.

\begin{Prop}\cite[Theorem 8.3]{Stembridge}
For given $\lambda, \mu \in DP$ we have
$$Q_{\lambda/\mu} = \sum_{\nu \in DP}{f^{\lambda}_{\mu \nu}} Q_{\nu},$$
where $f^{\lambda}_{\mu \nu}$ is the number of amenable tableaux $T$ of shape $D_{\lambda/\mu}$ and content $\nu$.
\end{Prop}

For $\lambda \in DP$, the corresponding Schur $P$-function is defined by $P_{\lambda} := 2^{-\ell(\lambda)} Q_{\lambda}$.
In  \cite[Chapter 8]{Stembridge}, Stembridge showed that the numbers $f^{\lambda}_{\mu \nu}$ above also appear in the product of $P$-functions:
$$P_{\mu} P_{\nu} = \sum_{\lambda \in DP}{f^{\lambda}_{\mu \nu} P_{\lambda}}\:.$$
Using this, one easily obtains the equation $f^\lambda_{\mu\nu} = f^\lambda_{\nu\mu}$ for all $\lambda, \mu, \nu \in DP$.\par

\begin{Def}
A \textbf{border strip} is a connected (skew) diagram $B$ such that for each box $(x,y) \in B$ the box $(x-1,y-1) \notin B$.
The box $(x,y) \in B$ such that $(x-1,y) \notin B$ and $(x,y+1) \notin B$ is called the \textbf{first box} of $B$.
The box $(u,v) \in B$ such that $(u+1,v) \notin B$ and $(u,v-1) \notin B$ is called the \textbf{last box} of $B$.\par
A (possibly disconnected) diagram $D$ where all components are border strips is called a \textbf{broken border strip}.
Then the first box of the rightmost component is called the first box of $D$, and the last box of the leftmost component is called the last box of $D$.\par
A $(p,q)$\textbf{-hook} is a set of boxes
$$\{(u,v+q-1), \ldots, (u,v+1), (u,v), (u+1, v), \ldots, (u+p-1, v)\}$$
for some $u, v \in \NN$.
\end{Def}
\begin{Def}
Let $T$ be a skew shifted tableau of shape $D_{\lambda/\mu}$.
Define $T^{(i)}$ by
$$T^{(i)} := \{(x,y) \in D_{\lambda/\mu} \mid |T(x,y)| = i\}.$$
\end{Def}
\begin{Def}\label{fitting}
Let $T$ be a tableau.
If the last box of $T^{(i)}$ is filled with $i$ we call $T^{(i)}$ \textbf{fitting}.
\end{Def}
\begin{Rem}
A restatement of \ref{amenable} (c) (respectively, \ref{amenable} (d)) is that $T^{(k)}$ (respectively, $T^{(k-1)}$) is fitting.
\end{Rem}

The remark before Theorem 13.1 of \cite{HoffmanHumphreys} states that for a given tableau $T$ of shape $D_{\lambda/\mu}$ the absolute values of the diagonals from top left to bottom right grow, that is $|T(x,y)| < |T(x+1,y+1)|$ for all $x,y$ such that $(x,y), (x+1,y+1) \in D_{\lambda/\mu}$.
An easy consequence is that each component of $T^{(i)}$ is a border strip.
This is mentioned in the part after Corollary 8.6 in \cite{SaganStanley} as well as the fact that each component of $T^{(i)}$ has two possible fillings with entries of $\{i', i\}$ which differ only by the marking of the entry in the last box of this component.\par
We will now prove a criterion for $k$-amenability of a tableau that avoids the use of the reading word and can be used as definition of $k$-amenability of tableaux.
The following technical definition is required to be able to state Lemma \ref{checklist}.

\begin{Def}
Let $\lambda, \mu \in DP$ and let $T$ be a tableau of $D_{\lambda/\mu}$.
Then
$$\mathcal{S}^{\boxtimes}_{\lambda/\mu}(x,y) := \{(u,v) \in D_{\lambda/\mu} \mid u \leq x, v \geq y\},$$
$$\mathcal{S}^{\boxtimes}_T(x,y)^{(i)} := S^{\boxtimes}_{\lambda/\mu}(x,y) \cap T^{-1}(i) \text{ where $T^{-1}(i)$ denotes the preimage of $i$},$$
$$\mathcal{B}_T^{(i)}:=\{(x,y) \in D_{\lambda/\mu} \mid T(x,y) = i' \text{ and } T(x-1,y-1) \neq (i-1)'\},$$
$$\widehat{\mathcal{B}_T^{(i)}}:=\{(x,y) \in D_{\lambda/\mu} \mid T(x,y) = i' \text{ and } T(x+1,y+1) \neq (i+1)'\}$$
and $b_T^{(i)} = |\mathcal{B}_T^{(i)}|$ for all $i$.
Then let $\mathcal{B}_T^{(i)}(d)$ denote the set of the first $d$ boxes of $\mathcal{B}_T^{(i)}$.
\end{Def}

\begin{Lem}\label{checklist}
Let $\lambda, \mu \in DP$ and $n := |D_{\lambda/\mu}|$.
Let $T$ be a tableau of $D_{\lambda/\mu}$.
Then the tableau $T$ is $k$-amenable if and only if either $c(T)_{k-1} = c(T)_k = 0$ or else it satisfies the following conditions:
\begin{enumerate}[(1)]
	\item $c(T)^{(u)}_{k-1} > c(T)^{(u)}_k$;
	\item when $T(x,y) = k$ then $|\mathcal{S}^{\boxtimes}_T(x,y)^{(k-1)}| \geq |\mathcal{S}^{\boxtimes}_T(x,y)^{(k)}|$;
	\item for each $(x,y) \in \mathcal{B}_T^{(k)}$ we have $|\mathcal{S}^{\boxtimes}_T(x,y)^{(k-1)}| > |\mathcal{S}^{\boxtimes}_T(x,y)^{(k)}|$;
	\item if $d := b_T^{(k)}+c^{(u)}_k-c^{(u)}_{k-1}+1 > 0$ then there is an injective map $\phi: \mathcal{B}_T^{(k)}(d) \rightarrow \widehat{\mathcal{B}_T^{(k-1)}}$ such that if $(x,y) \in \mathcal{B}_T^{(k)}(d)$ and $(u,v) = \phi(x,y)$ then for all $u < r < x$ we have $T(r,s) \notin \{k-1,k'\}$ for all $s$ such that $(r,s) \in D_{\lambda/\mu}$;
	\item $T^{(k-1)}$ is fitting;
  \item if $c(T)_k > 0$ then $T^{(k)}$ is fitting.
\end{enumerate}
\end{Lem}
\begin{proof}
First we want to show that tableaux that satisfy these conditions are indeed\linebreak $k$-amenable.
Clearly, such a tableau is $k$-amenable if $c(T)_k = c(T)_{k-1} = 0$.
Hence, we assume that $c(T)_k + c(T)_{k-1} \geq 1$.\par
Lemma \ref{checklist}~(2) ensures that if $w_i = k$ then we have $m_{k-1}(n-i) \geq\linebreak
|\mathcal{S}^{\boxtimes}_T(x(i),y(i))^{(k-1)}| > |\mathcal{S}^{\boxtimes}_T(x(i),y(i))^{(k)}|-1 = m_k(n-i)$ since $T(x(i)-1,y(i)-1) \neq k$ if $(x(i)-1,y(i)-1) \in D_{\lambda/\mu}$.
Lemma \ref{checklist}~(3) ensures that if $w_i = k'$ and $(x(i),y(i)) \in \mathcal{B}_T^{(k)}$ then $m_{k-1}(n-i) > m_k(n-i)$.
If $w_i = k'$ and $(x,y) := (x(i),y(i)) \notin \mathcal{B}_T^{(k)}$ then $T(x-1,y-1) = (k-1)'$.
But then $T(x-1,y) \in \{k', k-1\}$.
If $T(x-1,y) = k-1$ then we have $m_{k-1}(n-j+1) > m_k(n-j+1)$
where $j$ is defined by $(x(j),y(j)) = (x-1,y)$.
But then $m_{k-1}(n-i) > m_k(n-i)$.
If $T(x-1,y) = k'$ then either $(x-1,y) \in \mathcal{B}_T^{(k)}$ or $T(x-2,y-1) = (k-1)'$.
Then we can repeat this argument until we find a box $(z,y)$ where $z < x$ such that either $T(z,y) = k-1$ or $(z,y) \in \mathcal{B}_T^{(k)}$.
Thus, it is impossible to have $m_{k-1}(i) = m_k(i)$ and $w_{n-i} = k'$ for some $i$.
Hence, we showed that Definition \ref{amenable}~(a) is satisfied.\par
Lemma \ref{checklist}~(1) ensures that we always have $m_{k-1}(n) > m_k(n)$.
Let $i$ be such that $w_i = k'$, $T(x(i)-1,y(i)-1) = (k-1)'$ and $m_{k-1}(n+i-1) > m_k(n+i-1)$.
Then let $j$ be such that $(x(j),y(j)) = (x(i)-1,y(i)-1)$.
We have $m_{k-1}(n+i)\linebreak
\geq m_k(n+i)$ and $T(x,z) > k'$ for all $y < z \leq \lambda_x+x-1$ (the rightmost box of this row is $(x,\lambda_x+x-1)$).
Also, we have $T(x-1,w) < (k-1)'$ for all $\mu_{x-1}+x-1 \leq w < y$ (the leftmost box of this row is $(x-1,\mu_{x-1}+x-1)$).
Thus, we have $m_{k-1}(n+l) \geq m_k(n+l)$ for all $i \leq l \leq j-1$.
Then $m_{k-1}(n+j) \geq m_k(n+j)+1 > m_k(n+j)$.
Hence, Definition \ref{amenable}~(b) has not been violated between $w_i$ and $w_j$.
By this argument, $k$-amenability of $T$ depends on the boxes $(x,y) \in \mathcal{B}_T^{(k)}$.
If $w_i = k'$ and $(x(i),y(i))$ is one of the last $c^{(u)}_{k-1}-c^{(u)}_k-1$ boxes of $\mathcal{B}_T^{(k)}$ then $m_{k-1}(n+i) > m_k(n+i)$ since $m_{k-1}(n) = m_k(n)+c^{(u)}_{k-1}-c^{(u)}_k$.
Let $w_i = k'$ and let the box $(x(i),y(i)) \in \mathcal{B}_T^{(k)}(b_T^{(k)}+c^{(u)}_k-c^{(u)}_{k-1}+1)$.
By Lemma \ref{checklist}~(4), there is some $j$ such that $w_j = (k-1)'$ and $\phi(x(i),y(i)) = (x(j),y(j))$.
We have $m_{k-1}(n+i)-m_k(n+i) \geq c^{(u)}_{k-1}-c^{(u)}_k-(c^{(u)}_{k-1}-c^{(u)}_k-1)-1 = 0$ where the last $-1$ comes from the scanned entry $k'$ in the box $(x(i),y(i))$.
Note that pairs of boxes $(s,t)$ and $(s+1,t+1)$ such that $T(s,t) = (k-1)'$ and $T(s+1,t+1) = k'$ do not change the difference $m_{k-1}(n+i) - m_k(n+i)$ because the letter $w_i = k'$ cannot be between these entries in the reading word and, hence, both letters of such pairs are scanned before we scan $w_i = k'$.
Also for every box $(v,w) \in \mathcal{B}_T^{(k)}(b_T^{(k)}+c^{(u)}_k-c^{(u)}_{k-1}+1)$ such that $v > x(i)$ Lemma \ref{checklist}~(4) ensures that $\phi(v,w)$ is not in a row above the $x(i)^{\textrm{th}}$ row or in the $x(i)^{\textrm{th}}$ row to the right of $(x(i),y(i))$.
Hence, $T(v,w) = k'$ and $T(\phi(v,w)) = (k-1)'$ are scanned before $w_i=k'$ and these entries do not change the difference $m_{k-1}(n+i) - m_k(n+i)$.
If $x(j) \geq x(i)$ then $m_{k-1}(n+i)-m_k(n+i) > 0$ because $w_j = (k-1)'$ is scanned before $w_i = k'$.
If $x(j) < x(i)$ and $m_{k-1}(n+i)-m_k(n+i) = 0$ then $w_l \notin \{k-1,k'\}$ for all $i < l < j$.
Thus, there is no $i$ such that $m_{k-1}(n+i-1) = m_k(n+i-1)$ and $w_i \in \{k-1,k'\}$.
Hence, we showed that Definition \ref{amenable}~(b) is satisfied.\par
Lemma \ref{checklist}~(5) and Lemma \ref{checklist}~(6) are restatements of Definition \ref{amenable}~(c) and Definition \ref{amenable}~(d), respectively (as mentioned in the remark after Definition \ref{fitting}).
In total these conditions ensure $k$-amenability.\par
Now we want to show that if one of these conditions is not satisfied then $T$ is not $k$-amenable.
We may assume that $a+b > 0$.\par
Suppose Lemma \ref{checklist}~(1) is not satisfied.
Then we have $m_{k-1}(n) \leq  m_k(n)$ which contradicts Lemma \ref{unmarked}.\par
Suppose Lemma \ref{checklist}~(2) is not satisfied.
Let $i$ be such that $w_i = k$ is the first scanned entry $k$ such that $(x,y) := (x(i),y(i))$ violates Lemma \ref{checklist}~(2).
Then $T(x-1,y) \neq k-1$ and $|\mathcal{S}^{\boxtimes}_T(x,y)^{(k-1)}| = |\mathcal{S}^{\boxtimes}_T(x,y)^{(k)}|-1$.
We have to distinguish the cases $T(x-1,y-1) \neq k-1$ and $T(x-1,y-1) = k-1$.
If $T(x-1,y-1) \neq k-1$ then $m_{k-1}(n-i) = |\mathcal{S}^{\boxtimes}_T(x,y)^{(k-1)}| = |\mathcal{S}^{\boxtimes}_T(x,y)^{(k)}|-1 = m_{k}(n-i)$ and $w_i = k$ which violates Definition \ref{amenable}~(a).
If $T(x-1,y-1) = k-1$ then $T(x-1,y) = k'$ and, therefore, $T(x,y+1) \neq k$.
Then for $j$ such that $(x(j),y(j)) = (x-1,y)$ we must have $m_{k-1}(n-j) = m_k(n-j)$.
But then we have $m_{k-1}(n-j) = m_k(n-j)$ and $w_j = k'$ which also violates Definition \ref{amenable}~(a).\par
Suppose Lemma \ref{checklist}~(3) is not satisfied.
Let $(x,y) \in \mathcal{B}_T^{(k)}$ be such that\linebreak
$|\mathcal{S}^{\boxtimes}_T(x,y)^{(k-1)}| \leq |\mathcal{S}^{\boxtimes}_T(x,y)^{(k)}|$.
If $T(x-1,y-1) = k-1$ then if $(x,y-1) \in D_{\lambda/\mu}$ we have that $k-1 = T(x-1,y-1) < T(x,y-1) < T(x,y) = k'$ which is impossible.
Hence $(x,y-1) \notin D_{\lambda/\mu}$ and $x = y$.
But then $(x,y) = (x,x)$ is the lowermost leftmost box of $T^{(k)}$ and, since $T(x,x) = k'$, this means that $T^{(k)}$ is not fitting which violates Definition \ref{amenable}~(d).
Thus, there is no box $(x,y) \in \mathcal{B}_T^{(k)}$ such that $T(x-1,y-1) = k-1$.
Hence, if $i$ is such that $(x,y) = (x(i),y(i))$ then $m_{k-1}(n-i) \leq m_k(n-i)$.
If $m_{k-1}(n-i) < m_k(n-i)$ then $T$ is not $k$-amenable.
If $m_{k-1}(n-i) = m_k(n-i)$ then $w_i = k'$ which violates Definition \ref{amenable}~(a).\par
Suppose Lemma \ref{checklist}~(4) is not satisfied.
Thus, $b_T^{(i)}+c^{(u)}_k-c^{(u)}_{k-1}+1 > 1$ and there is a box $(x,y) \in \mathcal{B}_T^{(k)}(b_T^{(i)}+c^{(u)}_k-c^{(u)}_{k-1}+1)$ such that each box of $\mathcal{B}_T^{(k)}(b_T^{(i)}+c^{(u)}_k-c^{(u)}_{k-1}+1)$ that is below the $x^{\textrm{th}}$ row can be mapped to a different box with the given property of Lemma \ref{checklist}~(4) but $(x,y)$ cannot be mapped in this way.
If $i$ is such that $(x,y) = (x(i),y(i))$ then $m_{k-1}(n+i) = m_k(n+i)$ since $m_{k-1}(n+i) - m_k(n+i) = c^{(u)}_{k-1}-c^{(u)}_k-(b_T^{(i)}-(b_T^{(i)}+c^{(u)}_k-c^{(u)}_{k-1}+1))-1 = 0$ and, again, pairs of boxes $(s,t)$ and $(s+1,t+1)$ such that $T(s,t) = (k-1)'$ and $T(s+1,t+1) = k'$ do not change the difference $m_{k-1}(i) - m_k(i)$ as well as each box $(v,w) \in \mathcal{B}_T^{(k)}(b_T^{(k)}+c^{(u)}_k-c^{(u)}_{k-1}+1)$ such that $v > x$ that can be mapped to a different box with the given property of Lemma \ref{checklist}~(4) since $T(u,v) = k'$ and $T(\phi(u,v)) = (k-1)'$ are both scanned before the letter $w_i = k'$.
Since the box $(x,y)$ cannot be mapped to a box with the given property of Lemma \ref{checklist}~(4), this means that either there is some $l > i$ such that $m_{k-1}(n+l-1) = m_k(n+l-1)$ and $w_l \in \{k-1,k'\}$, which violates Definition \ref{amenable}~(b), or we have $m_{k-1}(n-i) = 0$ and $w_i = T(x(i),y(i)) = T(x,y) = k'$ which violates Definition \ref{amenable}~(a).\par
It is clear by definition that a tableau is not $k$-amenable if Lemma \ref{checklist}~(5) and Lemma \ref{checklist}~(6) are not satisfied.\par
Thus, we showed that the $k$-amenable tableaux are precisely the ones that satisfy the conditions in Lemma \ref{checklist}.
\end{proof}
\begin{Ex}
Let
$${\Yvcentermath1 T = \young(\times \times \times \times \times \times \times \times \meins 11,:\times \times \times \times \times \times \meins \mzwei 2,::\times \times \times \times \times 1,:::\times \times \times \times \mzwei ,::::\times \meins 12,:::::1\mzwei ,::::::2)}$$
be a tableau of shape $D_{(11,9,6,5,4,2,1)/(8,6,5,4,1)}$.
We will check the conditions of Lemma \ref{checklist} for $k = 2$ in the following.
We have $c(T)^{(u)}_1 = 5 > 3 = c(T)^{(u)}_2$.
Since $T^{-1}(2) = \{(2,10), (5,8), (7,7)\}$, we need to check condition (2) of Lemma \ref{checklist} for these boxes.
We have $|\mathcal{S}^{\boxtimes}_T(2,10)^{(1)}| = 2 \geq 1 = |\mathcal{S}^{\boxtimes}_T(2,10)^{(2)}|$, $|\mathcal{S}^{\boxtimes}_T(5,8)^{(1)}| = 3 \geq 2 = |\mathcal{S}^{\boxtimes}_T(5,8)^{(2)}|$ and $|\mathcal{S}^{\boxtimes}_T(7,7)^{(1)}| = 4 \geq 3 = |\mathcal{S}^{\boxtimes}_T(7,7)^{(2)}|$.
Since $\mathcal{B}_T^{(2)} = \{(2,9), (4,8)\}$, we need to check condition (3) of Lemma \ref{checklist} for these boxes.
We have $|\mathcal{S}^{\boxtimes}_T(2,9)^{(1)}| = 2 > 1 = |\mathcal{S}^{\boxtimes}_T(2,9)^{(2)}|$ and $|\mathcal{S}^{\boxtimes}_T(4,8)^{(1)}| = 3 > 1 = |\mathcal{S}^{\boxtimes}_T(4,8)^{(2)}|$.
Since $d := 2+3-5+1 = 1$, we have to find a map as in condition (4) of Lemma \ref{checklist} for the box $(2,9)$.
Such a map is $\phi((2,9)) = (2,8)$.
Another one is $\phi((2,9)) = (1,9)$.
Clearly, $T^{(1)}$ and $T^{(2)}$ are fitting.
Hence, the tableau $T$ is $2$-amenable.
\end{Ex}

It is easy to check that the conditions in the following corollary are included in the conditions of Lemma \ref{checklist}.
\begin{Co}\label{checklistco}
Let $\lambda, \mu \in DP$.
Let $T$ be a tableau of shape $D_{\lambda/\mu}$ such that either $c(T)_k = c(T)_{k-1} = 0$ or else it satisfies the following conditions:
\begin{enumerate}[(1)]
	\item there is some box $(x,y)$ such that $T(x,y) = k-1$ and $T(z,y) \neq k$ for all $z > x$;
	\item if $T(x,y) = k$ then there is some $z < x$ such that $T(z,y) = k-1$;
	\item if $T(x,y) = k'$ then $T(x-1,y-1) = (k-1)'$;
	\item $T^{(k-1)}$ is fitting;
  \item if $c(T)_k > 0$ then $T^{(k)}$ is fitting.
\end{enumerate}
Then the tableau is $k$-amenable.
\end{Co}

In the following, we will need a specific tableau $T_{\lambda/\mu}$ for a given diagram $D_{\lambda/\mu}$ that was constructed by Salmasian.

\begin{Def}\cite[before Lemma 3.5]{Salmasian}\label{Salmasian'salgorithm}
Let $D_{\lambda/\mu}$ be a skew diagram.
The tableau $T_{\lambda/\mu}$ is determined by the following algorithm:
\begin{enumerate}[(1)]
	\item Set $k = 1$ and $U_1(\lambda/\mu) = D_{\lambda/\mu}$.
	\item Set $P_k = \{(x,y) \in U_k(\lambda/\mu) \mid (x-1,y-1) \notin U_k(\lambda/\mu)\}$.
	\item For each $(x,y) \in P_k$ set $T_{\lambda/\mu}(x,y) = k'$ if $(x+1,y) \in P_k$, otherwise set $T_{\lambda/\mu}(x,y) = k$.
	\item Let $U_{k+1}(\lambda/\mu) = U_k(\lambda/\mu)\setminus P_k$.
	\item Increase $k$ by one, and go to (2).
\end{enumerate}
We also define $P_i(\lambda/\mu):=P_i$, for $i\ge 1$.
\end{Def}
\begin{Ex}
For $\lambda = (6,5,3,2)$ and $\mu = (4,1)$ we have
$$T_{\lambda/\mu} = {\Yvcentermath1 \young(::\meins 1,\meins 112,1\mzwei 2,:23).}$$
\end{Ex}
Salmasian proved the amenability of $T_{\lambda/\mu}$ in \cite[Lemma 3.9]{Salmasian}.
Also Corollary \ref{checklistco} proves amenability of $T_{\lambda/\mu}$ for all $\lambda, \mu \in DP$ since $T_{\lambda/\mu}$ is a prototype of a tableau satisfying the conditions of that corollary.

\section{Properties of the decomposition of skew Schur $Q$-functions}

Before we can start to answer the question which skew Schur $Q$-functions are $Q$-homogeneous we need to prove some results that will be important in the next section.\par
We want to decompose $Q_{\lambda/(n)}$ for any $1 \leq n \leq \lambda_1$, in particular for the case $n = \lambda_1-1$, which will be applied in Theorem \ref{notcon}.

\begin{Def}
Let $\lambda \in DP$.
Then the \textbf{border} is defined by
$$B_{\lambda} := \{(x,y) \in D_{\lambda} \mid (x+1,y+1) \notin D_{\lambda}\}.$$
Define $B_{\lambda}^{(n)} := \{D_{\lambda/\mu} \mid D_{\lambda/\mu} \subseteq B_{\lambda} \text{ and } |D_{\lambda/\mu}| = n\}$.
\end{Def}

\begin{Def}\label{E}
Let $\lambda \in DP$.
Define $E_{\lambda}$ to be the set of all partitions whose diagram is obtained after removing a corner in $D_{\lambda}$.
\end{Def}

Using the set $B_{\lambda}^{(n)}$ we can describe the decomposition of $Q_{\lambda/(n)}$ in general; the set $E_{\lambda}$ is required for a simpler version of the special case $n = 1$.
We will describe and prove the decomposition of $Q_{\lambda/(n)}$
and state the special cases $n = \lambda_1-1$ and $n = 1$ explicitly.

\begin{Prop}\label{lambda/n}
Let $\lambda \in DP$ and $1 \leq n \leq \lambda_1$ be an integer.
Then
$$Q_{\lambda/(n)} = \sum_{D_{\lambda/\nu} \in B_{\lambda}^{(n)} \hspace{1ex} (D_{\nu} \subseteq D_{\lambda})}{2^{comp(D_{\lambda/\nu})-1} Q_{\nu}}.$$
In particular,
$$Q_{\lambda/(\lambda_1-1)} = \sum_{(x,y) \in B^{\times}_{\lambda}}{c^{(x,y)}_{B_{\lambda}} Q_{D_{\mu} \cup \{(x,y)\}}}$$
where $D_{\mu} = D_{\lambda} \setminus B_{\lambda}$, $B^{\times}_{\lambda} := \{(x,y) \in B_{\lambda} \mid (x-1,y) \notin B_{\lambda} \text{ and } (x,y-1) \notin B_{\lambda}\}$ and
$$c^{(x,y)}_{B_{\lambda}} = \begin{cases}
1 &\mbox{if $(x,y)$ is the first or last box of $B_{\lambda}$} \\
2 &\mbox{otherwise},
\end{cases}$$
and
$$Q_{\lambda/(1)} = \sum_{\nu \in E_{\lambda}}{Q_{\nu}}.$$
\end{Prop}
\begin{proof}
Since $f^{\lambda}_{(n) \nu} = f^{\lambda}_{\nu (n)}$, we need to look at tableaux of shape $D_{\lambda/\nu}$ and content $(n)$.
These $n$ entries from $\{1', 1\}$ must be in the boxes of $B_{\lambda}$.
Hence, $D_{\lambda/\nu} \in B_{\lambda}^{(n)}$.
Thus, the constituents of $Q_{\lambda/(n)}$ with a non-zero coefficient are $Q_{\nu}$ such that $D_{\lambda/\nu} \in B_{\lambda}^{(n)}$.\par
Each component of $D$ can have two fillings that differ by the marking of the entry of the last box.
By definition of amenability, the last box of $D_{\lambda/\nu}$ must contain a $1$.
Thus, for each component of $D_{\lambda/\nu}$ except for the first one there are two possibilities how to fill the last box, giving the stated coefficient.
\end{proof}

For $\lambda, \mu \in DP$ the \textbf{lexicographical order} $\leq$ in $DP$ is defined as follows: $\lambda \leq \mu$ whenever either $\lambda = \mu$ or there is some $k$ such that $\lambda_i = \mu_i$ for $1 \leq i \leq k$ and $\lambda_{k+1} < \mu_{k+1}$ where $\lambda_k := 0$ if $k > \ell(\lambda)$.\par
The tableau $T_{\lambda/\mu}$ is one of the tableaux of $D_{\lambda/\mu}$ which have the lexicographical largest content.
We are also able to describe how many other tableaux have the same content as $T_{\lambda/\mu}$ and, hence, we can give the coefficient of the constituent indexed by the lexicographically largest partition, that is $c(T_{\lambda/\mu})$.
We want to show both statements now.
In the following, the sets $P_i=P_i(\lambda/\mu)$ are always the sets
arising in the construction of $T_{\lambda/\mu}$.

\begin{Lem}\label{lexlargest}
We have $c(T) \leq c(T_{\lambda/\mu})$ for all amenable tableaux $T$ of shape $D_{\lambda/\mu}$.
However, if $c(T) = c(T_{\lambda/\mu})$ then $T^{(i)} = P_i$.
\end{Lem}
\begin{proof}
In order to obtain the lexicographically largest content of an amenable tableau of shape $D_{\lambda/\mu}$, we have to insert the maximal number of entries from $\{1', 1\}$ in $D_{\lambda/\mu}$, then the maximal number of entries from $\{2', 2\}$ etc.\par
If $|T(x,y)| = 1$ then $(x-1,y-1) \notin D_{\lambda/\mu}$.
The set of such boxes is $P_1$.
The algorithm of Definition \ref{Salmasian'salgorithm} fills these boxes only with entries from $\{1', 1\}$.
Then the entries from $\{2', 2\}$ must be filled in boxes $(x,y)$ such that $(x-1,y-1) \notin D_{\lambda/\mu} \setminus P_1$.
The set of such boxes is $P_2$ and the algorithm of Definition \ref{Salmasian'salgorithm} fills these boxes only with entries from $\{2', 2\}$.
Repeating this argument for all entries greater than $2$ implies the statement.
\end{proof}

\begin{Prop}\label{numoflexmax}
Let $D_{\lambda/\mu}$ be a diagram.
Let $\nu = c(T_{\lambda/\mu})$.
Then we have
$$f^{\lambda}_{\mu \nu} = \prod_{i = 1}^{\ell(\nu)}{2^{comp(P_i)-1}}.$$
\end{Prop}
\begin{proof}
Let $T$ be an amenable tableau of $D_{\lambda/\mu}$ with content $\nu$.
By Lemma \ref{lexlargest}, we have $T^{(i)} = P_i$.
Thus, such a tableau $T$ can differ from $T_{\lambda/\mu}$ only by markings of some entries.
For each $i$ each component $C_2, \ldots, C_{comp(P_i)}$ of $P_i$ can be filled in two different ways that differ by the marking of the last box.
By Definition \ref{amenable} (c) and (d), the component $C_1$ must be fitting.\par
Let $(x,y)$ be the last box of one of the components $C_2, \ldots, C_{comp(P_i)}$ and let $T(x,y) = i'$.
Then, by Corollary \ref{checklistco}, $T$ is amenable because in this case we have $(x-1,y-1), (x,y-1) \in P_{i-1}$ and, hence, then $T(x-1,y-1) = (i-1)'$.
Thus, for each component of $P_i$ except for the first one, there are two possibilities how to fill the last box and the statement follows.
\end{proof}

Before we can start to classify the $Q$-homogeneous skew Schur $Q$-functions we need to introduce one further operation on diagrams different from the well known transposition and rotation.
As transposition and rotation of a diagram with special properties, this operation keeps the corresponding Schur $Q$-function unaltered (see \cite{BarekatVanWilligenburg} for a proof of this fact for transposition and rotation).

\begin{Def}
Let $D$ be a diagram.
The \textbf{orthogonal transpose} of $D$ is obtained as follows: reflect the boxes of $D$ along the diagonal $\{(z,-z) \mid z \in \NN\}$.
Move this arrangement of boxes such that the top row with boxes is in the first row and the lowermost box of the leftmost column with boxes is part of the diagonal $\{(z,z) \mid z \in \NN\}$.
We denote the orthogonal transpose of a diagram by $D^{ot}$.
\end{Def}
\begin{Ex}
For
$${\Yvcentermath1 D = \young(::\none \none \none ,\none \none \none \none \none ,\none \none \none \none ,:\none \none \none ,::\none )}$$
we obtain
$${\Yvcentermath1 D^{ot} = \young(:::\none \none ,:\none \none \none \none ,\none \none \none \none \none ,:\none \none \none ,::\none \none )}.$$
\end{Ex}

\begin{Lem}\label{ot}
Let $D$ be a diagram of shape $D_{\lambda/\mu}$.
There is a content-preserving bijection between the tableaux of shape $D$ and the tableaux of shape $D^{ot}$.
In particular, we have $Q_{D} = Q_{D^{ot}}$.
\end{Lem}
\begin{proof}
Let $T$ be a tableau of shape $D_{\lambda/\mu}$.
Let $\nu := c(T)$ and let $n := \ell(\nu)$.
Let $\Lambda$ be the map that maps $T$ to $\Lambda(T)$ where $\Lambda(T)$ is obtained as follows:
\begin{itemize}
	\item Reflect and move the boxes of $T$ together with their entries along the diagonal $\{(z,-z) \mid z \in \NN\}$. Denote the resulting filling of $D_{\lambda/\mu}^{ot}$ by $\bar{T}$.
	\item For all $i$ do the following:
	\begin{itemize}
		\item If $\bar{T}(x,y) \in \{i',i\}$ and $\bar{T}(x+1,y) \in \{i',i\}$ then set $\Lambda(T)(x,y) = (n-i+1)'$.
		\item If $\bar{T}(x,y) \in \{i',i\}$ and $\bar{T}(x,y-1) \in \{i',i\}$ then set $\Lambda(T)(x,y) = n-i+1$.
		\item If $\bar{T}(x,y) \in \{i',i\}$ and neither $\bar{T}(x+1,y) \in \{i',i\}$ nor $\bar{T}(x,y-1) \in \{i',i\}$ then if $(x,y)$ is the $k^{\textrm{th}}$ such box counted from the left let $(u,v)$ be the last box of the $k^{\textrm{th}}$ component of $T^{(i)}$. If $T(u,v) = i'$ set $\Lambda(T)(x,y) = (n-i+1)'$ and if $T(u,v) = i$ set $\Lambda(T)(x,y) = n-i+1$.
	\end{itemize}
\end{itemize}
One can see that $\Lambda$ maps tableaux of shape $D$ to tableaux of shape $D^{ot}$.
After orthogonal transposition, the rows and columns are weakly increasing since we orthogonally transpose the rows and columns and change the entries in reverse order.
Clearly, in $\Lambda(T)$ there is at most one $i$ in each column and at most one $i'$ in each row.
Hence, the properties of Definition \ref{tableaudef} are satisfied.\par
Let $a$ be the unmarked version of the least entry from $T$ and $b$ be the unmarked version of the greatest entry from $T$.
Then
$$c(\Lambda(T)) = \bar{\nu} = (\nu_1, \nu_2 \ldots, \nu_{a-1}, \nu_b, \nu_{b-1}, \nu_{b-2}, \ldots, \nu_{a+1}, \nu_a)$$
where $\nu_1 = \nu_2 = \ldots = \nu_{a-1} = 0$.\par
The map $\Lambda$ is an involution in the set of tableaux and, hence, a bijection.\par
Since $Q_{\lambda/\mu}$ is a symmetric function, there are as many tableaux of shape $D_{\lambda/\mu}$ with content $\nu$ as there are with content $\bar{\nu}$.
Thus, there is a bijection taking tableaux of $D_{\lambda/\mu}$ with content $\nu$ to tableaux of $D_{\lambda/\mu}$ with content $\bar{\nu}$.
Let $\Theta$ be such a bijection.
Then $\Omega := \Theta \circ \Lambda$ is a content-preserving bijection since $\Omega$ is a composition of bijections and each of these two bijections flips the content.
\end{proof}
\begin{Rem}
After proving the previous lemma the author discovered that DeWitt proved this result in \cite[section 4.2]{DeWitt} in a slightly different way (with a minor mistake in the content of the image of a tableau).
In her thesis she called this operation ``flip'' and stated that this operation is well known but unfortunately did not give a reference.
\end{Rem}
For ``unshifted'' diagrams, that is, diagrams $D_{\lambda/\mu}$ where $\ell(\mu) = \ell(\lambda)-1$, orthogonal transposition is just the concatenation of transposition and rotation.
But unlike transposition and rotation this operation also works for diagrams that are not unshifted.\par
It is possible to give bijections like in Lemma \ref{ot} for transposition and rotation of diagrams.
An example of a bijection for the rotation is to rotate the diagram together with the entries through 180 degrees and then change the entries in the same way as $\Lambda$ does.

\section{Classification of $Q$-homogeneous skew Schur $Q$-functions}

\begin{Def}
A symmetric function $f$ is called \textbf{$Q$-homogeneous} if it is some multiple of a single Schur $Q$-function, that is, if $f = k \cdot Q_{\nu}$ for some $\nu \in DP$ and some $k \in \NN$.
A diagram $D$ is called \textbf{$Q$-homogeneous} if the skew Schur $Q$-function $Q_D$ is $Q$-homogeneous.
\end{Def}

We are interested in answering the question which $Q_{\lambda/\mu}$ are $Q$-homogeneous, that is, for which $\lambda, \mu \in DP$ we have $Q_{\lambda/\mu} = k \cdot Q_{\nu}$.
Clearly, then we must have $\nu = c(T_{\lambda/\mu})$.
Proposition \ref{numoflexmax} gives a restriction on the number $k$ in $Q_{\lambda/\mu} = k \cdot Q_{\nu}$, namely it has to be some power of $2$.\par
In the following, we set again $P_i=P_i(\lambda/\mu)$ for $i\ge 1$; hence $P_i = T_{\lambda/\mu}^{(i)}$.
Note that we will always assume that $D_{\lambda/\mu}$ is basic.

\subsection{The disconnected case}

In the following we will find a classification of the $Q$-homogeneous skew Schur $Q$-functions indexed by a disconnected diagram.
We will first exclude all non-$Q$-homogeneous skew Schur $Q$-function indexed by a disconnected diagram, and then in Proposition \ref{notcon} we will prove the $Q$-homogeneity of the skew Schur $Q$-functions indexed by one of the remaining disconnected diagrams.

\begin{Lem}\label{hom1}
Let $comp(D_{\lambda/\mu}) > 1$ and $\nu = c(T_{\lambda/\mu})$.
If there is a component $C_i$ such that $i > 1$ and $C_i$ has at least two boxes then $f^{\lambda}_{\mu \bar{\nu}} > 0$ where $\bar{\nu} = (\nu_1-1, \nu_2+1,\linebreak
\nu_3, \nu_4, \ldots)$.
In particular, $Q_{\lambda/\mu}$ is not $Q$-homogeneous.
\end{Lem}
\begin{proof}
We may consider the case that a component which is not the first component has boxes in two rows.
Otherwise we may consider the orthogonal transpose of the diagram.\par
Let $C_i$ where $i > 1$ be a component that has boxes in at least two rows.
If $(x,y)$ is the rightmost box of the lowermost row of $C_i \cap P_1$ then $(x-1,y) \in P_1$ and $(x+1,y+1) \notin D_{\lambda/\mu}$.
We obtain a new tableau $T$ if we set $T(x,y) = 2$, $T(x-1,y) = 1$ and $T(r,s) = T_{\lambda/\mu}(r,s)$ for every other box $(r,s) \in D_{\lambda/\mu}$.
By Corollary \ref{checklistco}, $T$ is amenable and has content $c(T) = (\nu_1-1, \nu_2+1, \nu_3, \nu_4, \ldots)$.
\end{proof}
\begin{Ex}
For ${\Yvcentermath1 T_{\lambda/\mu} = \young(::::\meins 1,:::112,\meins 11,1\mzwei 2,:2)}$ we obtain ${\Yvcentermath1 T = \young(::::11,:::122,\meins 11,1\mzwei 2,:2)}$.
\end{Ex}

\begin{Lem}\label{hom2}
Let $comp(D_{\lambda/\mu}) > 2$ and $\nu = c(T_{\lambda/\mu})$.
Then we have $f^{\lambda}_{\mu \bar{\nu}} > 0$ where $\bar{\nu} = (\nu_1-1, \nu_2+1, \nu_3, \nu_4, \ldots)$.
In particular, $Q_{\lambda/\mu}$ is not $Q$-homogeneous.
\end{Lem}
\begin{proof}
Let $(x,y)$ be the rightmost box of the lowermost row of $C_2 \cap P_1$.
We obtain a new tableau $T$ if we set $T(x,y) = 2$ and $T(r,s) = T_{\lambda/\mu}(r,s)$ for every other box $(r,s) \in D_{\lambda/\mu}$.
By Corollary \ref{checklistco}, $T$ is $m$-amenable for $m > 2$.
There is a $2$ but no $1$ in the $y^{\textrm{th}}$ column.
However, there is a $1$ in the last box of $C_3 \cap P_1$.
Hence, by Lemma \ref{checklist}, amenability follows.
It is clear that $c(T) = (\nu_1-1, \nu_2+1, \nu_3, \nu_4, \ldots)$.
\end{proof}
\begin{Ex}
For ${\Yvcentermath1 T_{\lambda/\mu} = \young(:::::1,:::\meins 1,:::12,\meins 11,1\mzwei 2,:2)}$ we obtain ${\Yvcentermath1 T = \young(:::::1,:::\meins 1,:::22,\meins 11,1\mzwei 2,:2)}$.
\end{Ex}

\begin{Lem}\label{hom4}
Let $comp(D_{\lambda/\mu}) > 1$ and $\nu = c(T_{\lambda/\mu})$.
Suppose the leftmost column of $C_1$ (which is the leftmost column of $D_{\lambda/\mu}$) contains at least two boxes.
Then $f^{\lambda}_{\mu \bar{\nu}} > 0$ where $\bar{\nu} = (\nu_1-1, \nu_2, \nu_3, \ldots, \nu_z, \nu_{z+1}+1, \nu_{z+2}, \ldots)$ where $z := \ell(\lambda)-\ell(\mu)$.
In particular, $Q_{\lambda/\mu}$ is not $Q$-homogeneous.
\end{Lem}
\begin{proof}
Let $(x,x)$ be the last box of $P_1$.
We obtain a new tableau $T$ if we set $P_1' := P_1 \setminus \{(x,x)\}$ and use this instead of $P_1$ in the algorithm of Definition \ref{Salmasian'salgorithm}.
Let $P'_i := T^{(i)}$.
It is clear that $(x,x)$ is the last box of $P'_2$.
If $(x+1,x+1)$ is the last box of $P_2$ then $(x+1,x+1)$ is the last box of $P'_3$, etc.
Thus, the $P'_i$s are distinguished from the $P_i$s by at most one moved or added box.
By Corollary \ref{checklistco}, $T$ is $m$-amenable for $m > 2$.
There is a $1$ with no $2$ below in the last box of $C_2 \cap P_1$.
Thus, by Corollary \ref{checklistco}, $T$ is $2$-amenable and, hence, amenable.\par
It is clear that $c(T)_1 = \nu_1-1$ since $|P_1'| = |P_1|-1$.
The $P_i$s for all $2 \leq i \leq z$ satisfy the property that the last box is part of the main diagonal $\{(a,a) \mid a \in \NN\}$.
As mentioned above, they differ from $P'_i$s by the fact that the last box is not $(x+i-1,\linebreak
x+i-1)$ but instead $(x+i-2,x+i-2)$.
Thus, $|P'_i| = \nu_i$.
Then $(x+z-1,x+z-1)$ is the last box of $P'_{z+1}$ but since $(x+z,x+z) \notin D_{\lambda/\mu}$, it follows $|P'_{z+1}| = \nu_{z+1}+1$.
Hence, the content is $c(T) = (\nu_1-1, \nu_2, \nu_3, \ldots, \nu_z, \nu_{z+1}+1, \nu_{z+2}, \ldots)$.
\end{proof}
\begin{Ex}
For ${\Yvcentermath1 T_{\lambda/\mu} = \young(:::\meins 1,:::12,\meins 11,\meins \mzwei 2,1\mzwei 3,:2)}$ we obtain ${\Yvcentermath1 T = \young(:::\meins 1,:::12,\meins 11,1\mzwei 2,223,:3)}$.
\end{Ex}

\begin{Lem}\label{hom5}
Let $comp(D_{\lambda/\mu}) > 1$ and $\nu = c(T_{\lambda/\mu})$.
If $C_1$ has boxes above the row of the uppermost box of the leftmost column then $Q_{\lambda/\mu}$ is not $Q$-homogeneous.
\end{Lem}
\begin{proof}
Since Lemma \ref{hom4} states that diagrams which have more than one box in the leftmost column are not $Q$-homogeneous, it suffices to consider diagrams such that the leftmost column of $C_1$ has only one box.
Let $(t,r)$ be the rightmost box of $P_1$ in the lowermost row of $P_1$.
Note that the last box of $P_1$ is to the left of the $r^{\textrm{th}}$ column.
We obtain a new tableau $T$ if we modify the algorithm of Definition \ref{Salmasian'salgorithm} so that $P_1' := P_1 \setminus \{(t,r)\}$ is used instead of $P_1$ in the algorithm.\par
By Corollary \ref{checklistco}, the tableau $T$ is $m$-amenable for $m > 2$.
If $T(t,r) = 2$ then, by Corollary \ref{checklistco}, this tableau is $2$-amenable since $T(t-1,r) = 1$.
If $T(t,r) = 2'$ then we have $T(t-1,r-1) \neq 1'$ since $(t-1,r-1) \notin D_{\lambda/\mu}$.
However, there is a $1$ with no $2$ below it in the last box of $C_2 \cap P_1$.
Thus, by Lemma \ref{checklist}, this tableau is $2$-amenable and, hence, amenable.
Since $|P_1'| = |P_1|-1$,
the content satisfies $c(T) \neq \nu$.
\end{proof}
\begin{Ex}
For ${\Yvcentermath1 T_{\lambda/\mu} = \young(::::\meins 1,::::12,::\meins 1,111\mzwei ,:222)}$ we obtain ${\Yvcentermath1 T = \young(::::\meins 1,::::12,::11,11\mzwei 2,:223)}$.
\end{Ex}

\begin{Prop}\label{notcon}
Let $\lambda, \mu \in DP$ be such that $comp(D_{\lambda/\mu}) > 1$ and such that $D_{\lambda/\mu}$ is basic.
Then $Q_{\lambda/\mu} = k \cdot Q_{\nu}$ if and only if $k = 2$, $\lambda = (r+2,r,r-1, \ldots, 1)$, $\mu = (r+1)$ and $\nu = (r+1,r-1,r-2, \ldots, 1)$ for some $r \geq 1$.
\end{Prop}
\begin{proof}
Let $Q_{\lambda/\mu}$ be $Q$-homogeneous and $D_{\lambda/\mu}$ be a disconnected diagram.
Lemma \ref{hom1} states that for $1 < i \leq comp(D_{\lambda/\mu})$ every component $C_i$ can consist of only one box, and Lemma \ref{hom2} states that the diagram must consist of precisely two components.
Thus, $D_{\lambda/\mu}$ has only two components $C_1, C_2$ where $C_2$ consists of a single box.
Lemma \ref{hom4} implies that the leftmost column of $C_1$ can have only one box, and Lemma \ref{hom5} yields that this box is in the uppermost row of $C_1$.
This implies that $C_1$ has shape $D_{\alpha}$ for some $\alpha \in DP$.
The same must be true for the orthogonal transpose of the diagram.
Thus, $\alpha = (r, r-1, \ldots, 1)$ for some $r \geq 1$.
Therefore, we have $\lambda = (r+2, r, r-1, \ldots, 1)$ and $\mu = (r+1) = (\lambda_1-1)$.
By Proposition \ref{lambda/n}, $B^{\times}_{\lambda} = \{(1, r+1)\}$ and we obtain $\nu = (r+1,r-1,r-2, \ldots, 1)$ and $k = f^{\lambda}_{\mu \nu} = 2$.
\end{proof}
\begin{Rem}
The case $r = 1$ in Proposition \ref{notcon} also appeared in \cite[Theorem IV.3]{DeWitt} where $Q$-homogeneous skew Schur $Q$-functions with unshifted diagrams are considered.
\end{Rem}
\begin{Ex}
For $\lambda = (6,4,3,2,1)$ and $\mu = (5)$
the following two tableaux are the only amenable tableaux of shape $D_{\lambda/\mu}$:
$${\Yvcentermath1 \young(::::\meins ,1111,:222,::33,:::4), \hspace{1ex} \young(::::1,1111,:222,::33,:::4)}.$$
\end{Ex}

\subsection{The connected case}
We have finished the disconnected case and we now consider $Q$-homogeneous skew Schur $Q$-functions indexed by a connected diagram.
The following lemmas show the non-$Q$-homogeneity of $Q_{\lambda/\mu}$ if some $P_i$ in $T_{\lambda/\mu}$ has at least two components.
This leads to Lemma \ref{coeffegone} that shows that in this case for $Q_{\lambda/\mu} = k \cdot Q_{\nu}$ we obtain $k = 1$;  thus $D_{\lambda/\mu}$ is a ``strange'' diagram in the sense of Salmasian, classified by him in \cite{Salmasian}.
Hence this provides the classification of $Q$-homogeneous skew Schur $Q$-functions indexed by a connected diagram.

\begin{Lem}\label{comps1}
Let $D_{\lambda/\mu}$ be a diagram.
Let $\nu := c(T_{\lambda/\mu})$.
Let there be some $i > 1$ such that $comp(P_i) \geq 2$ and let $C_1, \ldots, C_{comp(P_i)}$ be the components of $P_i$.
Let $(x_l,y_l)$ and $(u_l,v_l)$ be the first box and the last box of $C_l$, respectively.
If for some $j \in \{1, 2, \ldots, comp(P_i)-1\}$ we have $v_{j+1} \geq y_j+2$ then $f^{\lambda}_{\mu \tilde{\nu}} > 0$ where $\tilde{\nu} = (\nu_1, \nu_2, \ldots, \nu_{i-2}, \nu_{i-1}-1, \nu_i+1, \nu_{i+1}, \nu_{i+2}, \ldots)$.
\end{Lem}
\begin{proof}
Let $(u,v) = (u_{j+1},v_{j+1})$.
Then $(u-1,v-1), (u,v-1) \in P_{i-1}$.
Let $(s,v-1)$ be the lowermost box of $P_{i-1}$ in the $(v-1)^{\textrm{th}}$ column.
We obtain a new tableau $T$ if we set $T(s,v-1) = i$, $T(s-1,v-1) = i-1$ and $T(r,t) = T_{\lambda/\mu}(r,t)$ for every other box $(r,t) \in D_{\lambda/\mu}$.
If $(s,v) \in D_{\lambda/\mu}$ then $T(s,v) = T_{\lambda/\mu}(s,v) \neq i'$ and the properties in Definition \ref{tableaudef} are satisfied.
By Corollary \ref{checklistco}, the tableau $T$ is amenable.
It is clear that $c(T)_{i-1} = \nu_{i-1}-1$ and $c(T)_i = \nu_i+1$ and $c(T)_k = \nu_k$ for $k \neq i-1, i$.
\end{proof}
\begin{Ex}
For $\lambda = (9,8,5,3,2)$ and $\mu = (6,5,2,1)$ the changes are written in boldface:
$${\Yvcentermath1 \young(::\meins 11,::\meins 22,\meins 11,\meins \mzwei ,12) \rightarrow \young(::\meins 11,::\deins 22,\meins 1\dzwei ,\meins \mzwei ,12)}.$$
\end{Ex}

\begin{Lem}\label{comps2}
Let $D_{\lambda/\mu}$ be a diagram.
Let $\nu := c(T_{\lambda/\mu})$ where $\nu_j := 0$ for $j > \ell(\nu)$.
Let there be some $i > 1$ such that $comp(P_i) \geq 2$ and let $C_1, \ldots, C_{comp(P_i)}$ be the components of $P_i$.
Let $(x_l,y_l)$ and $(u_l,v_l)$ be the first box and the last box of $C_l$, respectively.
If for some $j \in \{1, 2, \ldots, comp(P_i)-1\}$ we have $v_{j+1} = y_j+1$ then $f^{\lambda}_{\mu \bar{\nu}} > 0$ where $\bar{\nu} = (\nu_1, \nu_2, \ldots, \nu_{i-2}, \nu_{i-1}-1, \nu_i, \nu_{i+1}+1, \nu_{i+2}, \nu_{i+2}, \ldots)$.
\end{Lem}
\begin{proof}
Let $(x,y) = (x_j,y_j)$ and $(u,y+1) = (u_{j+1},v_{j+1})$.
Then $x > u$ and we have $(x-1,y), (x-2,y) \in P_{i-1}$.
Let $(s,y)$ be the lowermost box of $P_i$ in the $y^{\textrm{th}}$ column and let $t$ be such that $T_{\lambda/\mu}(t,y) = i-1$.
We obtain a new tableau $T$ if we set $T(a,y) = T_{\lambda/\mu}(a+1,y)$ for $t-1 \leq a \leq s-1$, $T(s,y) = (i+1)'$ if $(s+1,y) \in P_{i+1}$ or $T(s,y) = i+1$ if $(s+1,y) \notin P_{i+1}$, and $T(e,f) = T_{\lambda/\mu}(e,f)$ for every other box $(e,f) \in D_{\lambda/\mu}$.
If $(x-1,y+1) \in D_{\lambda/\mu}$ then $T_{\lambda/\mu}(x-1,y+1) \neq i'$, otherwise $T_{\lambda/\mu}(x,y+1) = i$ and the boxes of $C_k$ and $C_{k+1}$ are in the same component.\par
By Corollary \ref{checklistco}, $T$ is $m$-amenable for $m \neq i, i+1$.
There is possibly some $b$ such that $T(b,y) = i'$ and $T(b-1,y-1) \neq (i-1)'$.
However, there is some $c \geq b$ such that $T(c,y-1) = (i-1)'$ and $T(c+1,y) \neq i'$.
Thus, by Lemma \ref{checklist}, $i$-amenability follows.
We possibly have $T(s,y) = (i+1)'$ and $T(s-1,y-1) \neq i'$.
However, we have $T(u,y+1) = i$ and there is no $i+1$ in the $(y+1)^{\textrm{th}}$ column.
Hence, by Lemma \ref{checklist}, $(i+1)$-amenability follows.
It is clear that $c(T)_{i-1} = \nu_{i-1}-1$ and $c(T)_{i+1} = \nu_{i+1}+1$ and $c(T)_j = \nu_j$ for $j \neq i-1, i+1$.
\end{proof}
\begin{Ex}
For $\lambda = (11,10,9,5,4,3,2)$ and $\mu = (7,6,4,3)$ the changes are written in boldface:
$${\Yvcentermath1 \young(:::\meins 111,:::\meins \mzwei 22,::\meins 1233,::\meins \mzwei,111\mzwei ,:222,::33) \rightarrow \young(:::\meins 111,:::\deins \mzwei 22,::\meins \dmzwei 233,::\meins \dmzwei,111\dzwei ,:22\dmdrei ,::33)}.$$
\end{Ex}

If a diagram is connected then this implies that $P_1$ must be connected.
If in a connected diagram there is some $P_i$ where $i > 1$ that consists of at least two components then we can apply one of the two previous lemmas to show that the corresponding skew Schur $Q$-function is not $Q$-homogeneous.
Thus, for the $Q$-homogeneous skew Schur $Q$-functions the diagram $P_i$ must be connected for all $i$.
Using Lemma \ref{numoflexmax} we obtain the following lemma.

\begin{Lem}\label{coeffegone}
Let $Q_{\lambda/\mu} = k \cdot Q_{\nu}$ for some $k$.
If $comp(D_{\lambda/\mu}) = 1$ then $k = 1$.
\end{Lem}

\begin{Rem}
As mentioned above, Salmasian \cite{Salmasian} classified the ``strange'' diagrams $D_{\lambda/\mu}$  which are the ones satisfying $Q_{\lambda/\mu} = Q_{\nu}$ for some $\nu \in DP$.
Thus, the $Q$-homogeneous skew Schur $Q$-functions to connected diagrams
are those on the classification list in \cite[Theorem 3.2]{Salmasian}.
\end{Rem}

We have now completed the classification of the $Q$-homogeneous skew Schur $Q$-functions.
Since we are also interested in the only constituent of the decomposition, before we state our main Theorem \ref{homogeneous}, we need the following definition and lemma.

\begin{Def}
Let $\lambda = (\lambda_1, \lambda_2, \ldots, \lambda_{\ell(\lambda)}) \in DP$.
Let $\mu = (\lambda_{i_1}, \lambda_{i_2}, \ldots, \lambda_{i_{\ell(\mu)}})$ for $\{i_1, i_2, \ldots, i_{\ell(\mu)}\} \subseteq \{1, 2, \ldots, \ell(\lambda)\}$.
Then $\lambda \setminus \mu$ is defined as the partition obtained by removing the parts of $\mu$ from $\lambda$.
\end{Def}
\begin{Ex}
For $\lambda = (9,7,5,4,3,1)$ and $\mu = (5,3,1)$ we obtain $\lambda \setminus \mu = (9,7,4)$.
\end{Ex}

\begin{Lem}\label{caseb}
If $\lambda = (a, a-1, \ldots, 1)$ and $\mu$ is arbitrary then $Q_{\lambda/\mu} = Q_{\lambda \setminus \mu}$.
\end{Lem}
\begin{proof}
The diagram $D_{\lambda/\mu}^{ot}$ is equal to $D_{\alpha}$ for some $\alpha \in DP$.
Thus, by Lemma \ref{ot}, we have $Q_{\lambda/\mu} = Q_{D_{\lambda/\mu}^{ot}} = Q_{\alpha}$.\par
We obtain $\alpha$ as follows.
We will show that for all $1 \leq  k \leq a$ the number $k$ is either a part of $\alpha$ or a part of $\mu$  but it is never a part of both partitions.
For this proof only, we will not assume that $D_{\lambda/\mu}$ is basic.
This means that in this proof it is possible to have $\lambda_1 = \mu_1$.\par
The statement clearly holds for $\lambda = (1)$.
Let $\lambda = (a, a-1, \ldots, 1)$ with $a > 1$ and consider $D_{\lambda/\mu}$.\par
Case 1: $(1,a) \in \mu$.\par
Then $\mu_1 = a$ and the $a^{\textrm{th}}$ column of $D_{\lambda/\mu}$ has at most $a-1$ boxes.
Thus, $\alpha_1 < a$.
Let $U$ be the diagram obtained by removing the boxes of the first row.\par
Case 2: $(1,a) \notin \mu$.\par
Then $\mu_1 < a$ and the $a^{\textrm{th}}$ column of $D_{\lambda/\mu}$ has precisely $a$ boxes.
Thus, $\alpha_1 = a$.
Let $U$ be the diagram obtained by removing the boxes of the $a^{\textrm{th}}$ column.\par
In both cases we have $U = D_{\gamma/\beta}$ for $\gamma = (a-1,a-2, \ldots, 1)$ and some $\beta$.
By induction the statement follows.
\end{proof}
\begin{Ex}\label{casebex}
For $\lambda = (5,4,3,2,1)$ and $\mu = (5,3,2)$ the diagram is
$${\Yvcentermath1 \young(\times \times \times \times \times ,:\times \times \times \none ,::\times \times \none ,:::\none \none ,::::\none )},$$
where ${\Yvcentermath1 \young(\times )}$ denotes a box from $D_{\mu}$.\par
We want to calculate the partition $\alpha$ appearing in the equation $Q_{\lambda/\mu} = Q_{D_{\lambda/\mu}^{ot}} = Q_{\alpha}$.
Since $(1,5) \in D_{\mu}$, there cannot be 5 boxes in the first row of $D_{\lambda/\mu}^{ot} = D_{\alpha}$.
Thus, there is a part $5$ in $\mu$ but not in $\alpha$.
After removing the boxes of the first row we obtain
$${\Yvcentermath1 \young(\times \times \times \none ,:\times \times \none ,::\none \none ,:::\none )}.$$
We have $(1,4) \notin D_{\mu}$ and, thus, there is no part $4$ in $\mu$ but a part $4$ in $\alpha$.
After removing the fourth column we obtain
$${\Yvcentermath1 \young(\times \times \times ,:\times \times ,::\none )}.$$
We have $(1,3) \in D_{\mu}$ and, thus, there is no part $3$ in $\alpha$ but in $\mu$.
After removing the boxes of the first row we obtain
$${\Yvcentermath1 \young(\times \times ,:\none )}.$$
We have $(1,2) \in D_{\mu}$ and, thus, there is no part $2$ in $\alpha$ but in $\mu$.
After removing the boxes of the first row we obtain
$${\Yvcentermath1 \young(\none )}.$$
We have $(1,1) \notin D_{\mu}$ and, thus, there is no part $1$ in $\mu$ but a part $1$ in $\alpha$.\par
We obtain $\alpha = (4,1) = (5,4,3,2,1) \setminus (5,3,2)$.
\end{Ex}

\begin{Th}\label{homogeneous}
Let $\lambda, \mu \in DP$ such that $D_{\lambda/\mu}$ is basic.
We have $Q_{\lambda/\mu} = k \cdot Q_{\nu}$ if and only if $\lambda$, $\mu$, $\nu$ and $k$ satisfy one of the following properties:
\begin{enumerate}[(i)]
	\item $\lambda$ arbitrary, $\mu = \emptyset$ and $\nu = \lambda$ and $k = 1$,
	\item $\lambda = (r, r-1, \ldots, 1)$ and $0 < \ell(\mu) < r-1$ for some $m$ and $\nu = \lambda \setminus \mu$ and $k = 1$,
	\item $\lambda = (p+q+r, p+q+r-1, p+q+r-2, \ldots, p)$, $\mu = (q, q-1, \ldots, 1)$, where $p, q, r \geq 1$ and $\nu = (p+r+q, p+r+q-1, p+r+q-2, \ldots, p+q+1, p+q,\linebreak
	p+q-2, p+q-4, \ldots, \max\{p-q, q+2-p\})$ and $k = 1$,
	\item $\lambda = (p+q, p+q-1, p+q-2, \ldots, p+1, p)$, $\mu = (q, q-1, \ldots, 1)$, where $p, q \geq 1$ and $\nu = (p+q, p+q-2, p+q-4, \ldots, \max\{p-q, q-p+2\})$ and $k = 1$,
	\item $\lambda = (r+2,r,r-1, \ldots, 1)$, $\mu = (r+1)$ and $\nu = (r+1,r-1,r-2, \ldots, 1)$ for some $r \geq 1$ and $k = 2$.
\end{enumerate}
\end{Th}
\begin{proof}
Case (i) is trivial.
For the cases (ii), (iii) and (iv) the proof of homogeneity is the main work of \cite{Salmasian} and will not be repeated here.
In case (ii), by Lemma \ref{caseb}, we have $Q_{\lambda/\mu} = Q_{(\lambda/\mu)^{ot}} = Q_{\alpha}$ for $D_{\alpha} = D_{\lambda \setminus \mu}$.
The partition $\nu$ for the cases (iii) and (iv) are easy to deduce since each $P_i$ is a hook.
Case (v) was shown in Proposition \ref{notcon}.
\end{proof}

\begin{Ack}
The QF package for Maple made by John Stembridge (\url{http://www.math.lsa.umich.edu/~jrs/maple.html}) was a helpful tool for analysing the decomposition of skew Schur $Q$-functions.\par
This paper is based on the research I did for my master's thesis and my PhD thesis which were supervised by Prof. Christine Bessenrodt.
I am very grateful to Christine Bessenrodt that she introduced me to (skew) Schur $Q$-functions, and I would like to thank her for her help in writing this paper as well as for supervising my research.
\end{Ack}

\end{document}